\documentclass[a4paper]{amsart}
\usepackage{amsmath}
\usepackage{amssymb}
\usepackage{amsthm}
\usepackage{tikz}
\usepackage{verbatim}
\usepackage[justification=centering]{caption}
\usepackage{thmtools, thm-restate}
\declaretheorem[numberwithin=section]{theorem}

\newtheorem{proposition}[theorem]{Proposition}
\newtheorem{lemma}[theorem]{Lemma}
\newtheorem{corollary}[theorem]{Corollary}
\newtheorem*{theo}{Theorem}

\theoremstyle{definition}
\newtheorem{definition}[theorem]{Definition}

\numberwithin{equation}{section}

\newcommand{\R}{\mathbb{R}}
\newcommand{\Z}{\mathbb{Z}}

\newcommand{\Bb}{\mathcal{B}}

\newcommand{\DD}{\mathcal{D}}

\newcommand{\FF}{\mathcal{F}}

\newcommand{\Ss}{\mathcal{S}}

\newcommand{\Int}{\mathrm{Int}}

\newcommand{\Ce}{C_\varepsilon}
\newcommand{\De}{D_\varepsilon}
\newcommand{\Se}{S_\varepsilon}

\newcommand{\ca}{\hspace{-2pt}\raisebox{-2pt}{
\begin{tikzpicture} [scale=0.3]
 \draw (0,-0.2) -- (-0.4,1) (0.3,-0.2) -- (-0.1,1);
\end{tikzpicture}
}\hspace{-2pt}}

\definecolor{vert}{RGB}{0,205,0}

\begin{document}

\title{On diffeomorphisms of $4$--dimensional $1$--handlebodies}

\author{Delphine Moussard}

\begin{abstract}
 We give a new proof of Laudenbach and Poénaru's theorem, which states that any diffeomorphism of the boundary of a $4$--dimensional $1$--handlebody extends to the whole handlebody. Our proof is based on the classification of Heegaard splittings of double handlebodies and a result of Cerf on diffeomorphisms of the $3$--ball. Further, we extend this theorem to $4$--dimensional compression bodies, namely cobordisms between $3$--manifolds constructed using only $1$--handles: when the negative boundary is a product of a compact surface and an interval, we show that every diffeomorphism of the positive boundary extends to the whole compression body. This involves a strong Haken theorem for spheres and disks in sutured Heegaard splittings and a classification of sutured Heegaard splittings of double compression bodies. Finally, we show how this applies to the study of relative trisection diagrams for compact $4$--manifolds.
\end{abstract}

\maketitle

\section{Introduction}

A famous theorem of Laudenbach and Poénaru asserts that every diffeomorphism of the boundary of a $4$--dimensional $1$--handlebody extends to a diffeomorphism of the whole handlebody \cite{LP}. This result is of great importance in the theory of smooth $4$--manifolds because it implies that, given a handle decomposition of a closed $4$--manifold $X$, the attaching information for the $1$-- and $2$--handles contained in a Kirby diagram is sufficient to determine $X$ up to diffeomorphism. Likewise, in the theory of trisections, it implies that a trisection diagram determines a unique closed $4$--manifold up to diffeomorphism. We give here an alternative proof of Laudenbach--Poénaru's theorem, based on two main ingredients. The first one is the uniqueness of the minimal genus Heegaard splitting of the boundary of a $4$--dimensional $1$--handlebody, that is a double handlebody, due to Carvalho and Oertel \cite{CarOer}. The second ingredient is the fact that every diffeomorphism of a $3$--dimensional handlebody that restricts to the identity on the boundary is isotopic to the identity; this is based on a result of Cerf \cite{Cerf4}. From that point, our proof of Laudenbach--Poénaru's theorem is very short.

Carvalho and Oertel actually used much of the same machinery that Laudenbach and Poénaru used in their original proof. In particular, both papers relied on Laudenbach’s results from \cite{Laudenbach}. So at first glance this new proof might be regarded as a repackaging of Laudenbach and Poénaru’s original proof. However, in \cite{HS}, Hensel and Schultens reprove Carvalho--Oertel's result using brief cut and paste arguments. Thus, we believe that this proof of Laudenbach--Poénaru’s theorem represents a true simplification of the original.

We then extend the setting and consider $4$--dimensional compression bodies. Such a compression body is a cobordism between two $3$--manifolds, its negative boundary and its positive boundary, constructed using only $1$--handles. We further require that the negative boundary is a product of a compact surface and an interval. These compression bodies are the building blocks of the so-called relative trisections of compact $4$--manifolds with boundary. We generalize Laudenbach--Poénaru's theorem to compression bodies.

\begin{theo}[Theorem~\ref{th:LPrel}]
 Let $V$ be a $4$--dimensional compression body. Assume the negative boundary of $V$ is a product $P\times I$, where $P$ is a compact oriented surface which contains no $2$--sphere. Then every diffeomorphism of the positive boundary $\partial_+V$ of $V$ that is the identity on $\partial(\partial_+V)$ extends to a diffeomorphism of $V$.
\end{theo}

From this result, we recover the statement, due to Castro, Gay and Pinz\'on-Caicedo \cite{CGPC2}, that every relative trisection diagram determines a unique compact $4$--manifold up to diffeomorphism, and we extend it to the case when the page of the trisection (the surface $P$) is allowed to contain closed components. This fails when $P$ is allowed to contain $2$--spheres, as we show with an example which was communicated to us by David Gay.

The proof of Theorem~\ref{th:LPrel} follows the lines of our proof of Laudenbach--Poénaru's theorem. The second ingredient is easily adapted to the relative case: given a $3$--dimensional compression body $C$, namely a cobordism between two compact surfaces constructed using only $1$--handles, we show that every diffeomorphism of $C$ that restricts to the identity on its positive boundary is isotopic to the identity. The first ingredient requires more work. The positive boundary of a $4$--dimensional compression body is diffeomorphic to the connected sum of some products $F\times I$, where $F$ is a compact surface, and some copies of $S^1\times S^2$. We need to understand the sutured Heegaard splittings of these so-called double compression bodies.

\begin{theo}[Theorem~\ref{th:doublecompHSfull}]
 Every two sutured Heegaard splittings of a double compression body with the same genus are isotopic.
\end{theo}

This will essentially follow from the strong Haken theorem due to Scharlemann \cite{Scharlemann}: every $2$--sphere embedded in a $3$--manifold equipped with a Heegaard splitting is isotopic to a $2$--sphere which intersects the Heegaard surface along a single circle. Another proof of this result was given by Hensel and Schultens in \cite{HS}, efficiently using the sphere complex of the manifold. We adapt the proof of the original Haken's theorem and apply Hensel--Schultens' technique to obtain a strong Haken theorem for spheres and disks in sutured Heegaard splittings.

\begin{theo}[Corollary~\ref{cor:strongHaken}]
 In a strongly admissible sutured Heegaard splitting, every collection of disjoint essential embedded $2$--spheres and properly embedded disks, where the boundary of each disk meets each compression body along an arc, and these two arcs are both essential or both inessential, is isotopic to a collection of Haken spheres and disks.
\end{theo}

To prove our relative Laudenbach--Poénaru's theorem, we only need the uniqueness of the minimal genus Heegaard splittings of double compression bodies. However, the general classification result is also useful in the theory of relative trisections. In the original definition of a relative trisection, one has to precisely describe a decomposition of the boundary of the building blocks, namely the $4$--dimensional compression bodies, in order to prescribe the way the different pieces should meet. We give a simpler definition of a relative trisection, and Theorem~\ref{th:doublecompHSfull} shows that the two definitions are equivalent.

\subsection*{Plan of the paper}
In Section~\ref{secLP}, we reprove Laudenbach--Poénaru's theorem. In Section~\ref{secHS}, we define sutured Heegaard splittings, we give a proof of a strong Haken theorem in this setting, and we classify the Heegaard splittings of double compression bodies. In Section~\ref{secLPrel}, we prove our relative Laudenbach--Poénaru's theorem, we apply it to the study of relative trisection diagrams, and we discuss the case when there is a $2$--sphere in the page $P$.

\subsection*{Conventions}
The boundary of an oriented manifold with boundary is oriented using the outward normal first convention. 
If $M$ is an oriented manifold, $-M$ represents $M$ with the opposite orientation. 
If $M$ is a compact manifold and $N$ is a submanifold, we denote by $M\ca N$ the manifold ``$M$ cut along $N$'', which comes with a surjective map $\pi:M\ca N\to M$ such that $\pi$ is a diffeomorphism from $\pi^{-1}(M\setminus N)$ to $M\setminus N$ and a double cover from $\pi^{-1}(N)$ to $N$.

\subsection*{Acknowledgements}
I warmly thank Trenton Schirmer for many helpful conversations and for valuable comments on the first version of the paper. I am also grateful to David Gay for helpful conversations.

\section{Proof of Laudenbach--Poénaru's theorem via Heegaard splittings}
\label{secLP}

A {\em genus--$g$ handlebody} is a $3$--manifold diffeomorphic to a $3$--ball with $g$ $1$--handles glued on its boundary. A {\em Heegaard splitting} of a closed $3$--manifold $M$ is a decomposition $M=H_1\cup_\Sigma H_2$, where $H_1$ and $H_2$ are handlebodies and $\Sigma=\partial H_1=-\partial H_2$. 
A {\em genus--$g$ double handlebody} is a $3$--manifold diffeomorphic to~$\sharp_{i=1}^g(S^1\times S^2)$ (it is the result of gluing of two copies of a genus--$g$ handlebody along their boundaries via the identity map). 

The first preliminary result we need to reprove Laudenbach--Poénaru's theorem is the classification of minimal genus Heegaard splittings of double handlebodies. This result was originally proved by Carvalho and Oertel \cite[Theorem 1.4]{CarOer}; another proof by Hensel and Schultens is given in \cite[Proposition~1.2]{HS}. It is recovered in Theorem~\ref{th:doublecompHS} within the more general setting of double compression bodies.

\begin{theorem}[Carvalho--Oertel] \label{th:CO}
 A genus--$g$ double handlebody admits a unique genus--$g$ Heegaard splitting, up to isotopy.
\end{theorem}

The second preliminary result is based on the following theorem of Cerf \cite{Cerf4}. 

\begin{theorem}[Cerf]
 Every diffeomorphism of a $3$--ball that is the identity on the boundary is isotopic to the identity, relative to the boundary.
\end{theorem}

We shall generalize this fact to handlebodies of positive genus. 
A {\em defining disk system} for a $3$--dimensional handlebody $H$ is a union $\DD$ of disjoint properly embedded disks such that $H\ca\DD$ is a $3$--ball.

\begin{lemma}
 Let $H$ be a $3$--dimensional handlebody. Every two defining disk systems for $H$ that coincide on the boundary are isotopic.
\end{lemma}
\begin{proof}
 This follows from a standard innermost disk argument, using the fact that handlebodies are irreducible.
\end{proof}

\begin{lemma} \label{lemma:diffeohandlebody}
 Let $H$ be a $3$--dimensional handlebody. Let $\varphi$ be a diffeomorphism of $H$. If $\varphi$ is the identity on $\partial H$, then $\varphi$ is isotopic to the identity, relative to the boundary.
\end{lemma}
\begin{proof}
 Pick a defining disk system $\DD$ for $H$. Then $\varphi(\DD)$ is another defining disk system with the same boundary, thus isotopic to $\DD$. By the isotopy extension theorem, there is an ambient isotopy sending $\varphi(\DD)$ to $\DD$, keeping $\partial H$ fixed. Hence, up to isotoping $\varphi$, we can assume that $\varphi(\DD)=\DD$. Now, every diffeomorphism of a $2$--disk, which is the identity on the boundary, is isotopic to the identity (Smale, see \cite[p.132]{Cerf4}). 
 Hence we can even assume that $\varphi$ is the identity on $\partial H\cup\DD$. We are led to a diffeomorphism of a $3$--ball which is the identity on the boundary, and we apply Cerf's result. 
\end{proof}

A similar proof is given by Hensel in \cite[Lemma~3.1]{Hensel} to prove the analogous result on homeomorphisms, using the Alexander trick instead of the results of Smale and Cerf.

A {\em $4$--dimensional $1$--handlebody} is a compact oriented smooth $4$--manifold obtained from a $4$--ball by adding a finite number of $1$--handles. The number of $1$--handles glued is the {\em genus} of the handlebody. Note that the boundary of a $4$--dimensional $1$--handlebody of genus $g$ is a genus--$g$ double handlebody.

\begin{theorem}[Laudenbach--Poénaru] \label{th:LP}
 Let $Z$ be a $4$--dimensional $1$--handlebody. Then every diffeomorphism of $\partial Z$ extends to a diffeomorphism of $Z$.
\end{theorem}
\begin{proof}
 Fix an identification of $Z$ with a product $H\times I$, with $H$ a $3$--dimensional genus--$g$ handlebody, where the vertical boundary $\partial H\times I$ has been collapsed along the $I$--factor. This induces a foliation of $Z$ by $3$--dimensional handlebodies $H_t$, $t\in[0,1]$, which meet exactly along their common boundary $\Sigma=\partial H_t$. This surface $\Sigma$ defines a minimal genus Heegaard splitting of $\partial Z$.
 
 Now take a diffeomorphism $\varphi$ of $\partial Z$. It sends the Heegaard splitting $\partial Z=H_0\cup_\Sigma H_1$ onto another splitting $\partial Z=\varphi(H_0)\cup_{\varphi(\Sigma)} \varphi(H_1)$, with the same genus. By Theorem~\ref{th:CO}, they are isotopic. Let $(f_t)_{t\in[0,1]}$ be an isotopy of $\partial Z$, where $f_0$ is the identity of $\partial Z$, $f_1(\Sigma)=\varphi(\Sigma)$, and $f_1(H_t)=\varphi(H_t)$ for each $t=0,1$. Consider a collar neighborhood $\partial Z\times [0,1]$ of $\partial Z$, where $\partial Z$ is identified with $\partial Z\times\{0\}$. We extend $\varphi$ to this collar neighborhood by setting $\phi(x,t)=(f^{-1}_t\circ\varphi(x),t)$. The induced diffeomorphism of $\partial Z\times\{1\}$ preserves $\Sigma$, $H_0$ and $H_1$, and it remains to extend it to the complement of the collar neighborhood. This allows us to assume that $\varphi(\Sigma)=\Sigma$ and $\varphi(H_t)=H_t$ for each $t=0,1$. 
 
 Now $\varphi_{|H_0}$ and $\varphi_{|H_1}$ define two diffeomorphisms of $H$ which coincide on the boundary. By Lemma~\ref{lemma:diffeohandlebody}, there is an isotopy $\varphi_t$ of diffeomorphisms of $H$ from $\varphi_{|H_0}$ to $\varphi_{|H_1}$. The map $\phi:Z\to Z$ induced by the diffeomorphism $(x,t)\mapsto (\varphi_t(x),t)$ of $H\times I$ is the desired diffeomorphism.
\end{proof}

\section{Sutured Heegaard splittings}
\label{secHS}
\subsection{Compression bodies}
\begin{definition}
 A \emph{compression body} $C$ is a cobordism from a non-empty connected compact oriented surface $\partial_+C$ to a compact oriented surface $\partial_-C$ which is constructed using only $2$--handles and $3$--handles, where enough $3$--handles are glued to avoid any $S^2$--component in $\partial_-C$. A \emph{lensed} compression body is then obtained by collapsing the vertical boundary of the cobordism so that the boundary of $\partial_+C$ becomes identified with the boundary of $\partial_-C$. 
\end{definition}

Note that the definition includes the possibility that $\partial_-C$ be empty (in which case $\partial_+C$ is closed and it makes no difference for the compression body to be lensed). Note also that a compression body can alternatively be constructed by adding $1$--handles, either to a thickening of the negative boundary, which is a compact oriented surface containing no $2$--sphere, or to a $3$--ball. 
In what follows, compression bodies are supposed to be lensed.

\begin{definition}
 Let $C$ be a compression body. A \emph{defining disk system} for $C$ is a collection $\DD$ of disjoint disks properly embedded in $C$ such that $\partial\DD\subset\partial_+C$ and $C\ca\DD$ is a thickening of $\partial_-C$, or a $3$--ball if $\partial_-C$ is empty. The boundary $\partial\DD\subset\partial_+C$ is a {\em cut-system} for $C$.
\end{definition}

Note that defining disk systems do exist: take for instance the core disks of the $2$--handles in the definition (with a minimal number of $2$-- and $3$--handles).

\begin{lemma} \label{lemma:compirr}
 Every compression body is irreducible. 
\end{lemma}
\begin{proof}
 We start with a trivial compression body $P\times I$, where $P$ is a compact connected oriented surface different from $S^2$. We embed this product in $\R^3$, which is irreducible. If $S$ is a $2$--sphere embedded in~$P\times I$, then $S$ bounds a $3$--ball $B$ in $\R^3$. If $\partial P$ is non-empty, then $\partial(P\times I)$ is connected, so that it is contained in~$\R^3\setminus B$, and $B\subset(P\times I)$. Now assume $P$ is closed. If $S$ separates the two boundary components of $P\times I$, then the retraction of $P\times I$ onto $P\times\{0\}$ provides a map $f:S\to P$ which induces an isomorphism $f_*:H_2(S)\to H_2(P)$. But $f$ lifts to a map $\tilde f:S\to\tilde P$, where $\tilde P$ is the universal cover of $P$, which is non-compact since $g(P)>0$. It follows that $f_*$ factors through $H_2(\tilde P)=0$. We get a contradiction and conclude that $S$ does not separate the boundary of $P\times I$, so that $B\subset(P\times I)$.

 Now let $S$ be a $2$--sphere embedded in a compression body~$C$. Let $\DD$ be a defining disk system for~$C$. 
 By the previous case, the components of $C\ca\DD$ are irreducible. If $S$ meets $\DD$, choose an intersection curve $\gamma$ which is innermost in $S$. Then $\gamma$ bounds a disk in $\DD$ and a disk in $S$, which together form a $2$--sphere in $C\ca\DD$, and this sphere bounds a $3$--ball. Thus we can isotope $\DD$ to remove this intersection curve. Iterating, we can assume $S$ is disjoint from $\DD$ and we are done.
\end{proof}

Given a compact oriented $3$--manifold $M$, a {\em sutured decomposition} of $\partial M$ is a decomposition $\partial M=\partial_0 M\cup_s\partial_1 M$, where $\partial_0 M$ and $\partial_1 M$ are two compact surfaces containing no $2$--sphere and oriented like~$\partial M$, and $s$ is their common boundary oriented as $\partial(\partial_1 M)$. The curve $s$ is called the {\em suture set} and its components are the {\em sutures}. Note that, if $\partial_0 M$ and $\partial_1 M$ have no closed component, the datum of the oriented suture set fully determines the sutured decomposition. A {\em sutured $3$--manifold} is a manifold whose boundary is equipped with a sutured decomposition.

\begin{definition}
 Let $M$ be a sutured $3$--manifold. A {\em sutured Heegaard splitting} of $M$ is a decomposition $M=C_1\cup_\Sigma C_2$ where $C_1$ and $C_2$ are compression bodies, $\Sigma=\partial_+ C_1=-\partial_+ C_2$ is a compact connected oriented surface, $\partial_-C_1=\partial_0 M$ and $\partial_-C_2=\partial_1 M$.
\end{definition}

Note that the definition implies that the boundary of the Heegaard surface $\Sigma$ coincides with the suture set on $\partial M$. Throughout the article, we will always consider sutured Heegaard splittings of manifolds with a fixed sutured decomposition on the boundary. 
Every sutured $3$--manifold admits a sutured Heegaard splitting, unique up to stabilization (see Juhasz \cite[Section~2]{Juh} or Dissler \cite[Section~3]{Dissler1}).

\subsection{Haken's theorem for spheres and disks in sutured Heegaard splittings}

\begin{definition}
 An essential embedded $2$--sphere in a $3$--manifold with a Heegaard splitting is {\em Haken} if it intersects the Heegaard surface transversely along a connected simple closed curve.
\end{definition}

\begin{definition}
 An embedded $2$--sphere $S$ in a compact $3$--manifold $M$ is {\em surviving} if $S$ is essential in the manifold obtained from $M$ by filling every $S^2$--component in $\partial M$ with a $3$--ball.
\end{definition}

The following theorem is due to Haken in the closed case \cite{Haken}. We adapt the proof from Jaco \cite{Jaco} to the setting of sutured Heegaard splittings.

\begin{theorem}[Haken] \label{th:Haken}
 Let $M$ be a sutured $3$--manifold and $M=C_1\cup_\Sigma C_2$ a Heegaard splitting. If there is an essential ({\em resp} surviving) embedded $2$--sphere $S\subset M$, then there is an essential ({\em resp} surviving) embedded $2$--sphere $S'\subset M$ which is Haken.
\end{theorem}

We will need two lemmas.
A {\em spanning annulus} in a compression body $C$ is a properly embedded annulus with a boundary component on each of $\partial_{\pm}C$. Similarly, a {\em spanning disk} in $C$ is a properly embedded disk whose boundary meets each of $\partial_{\pm}C$ along a single arc.

\begin{lemma} \label{lemma:Haken1}
 Let $C$ be a compression body. Let $S$ be a genus--$0$ compact surface with no closed component. Assume $(S,\partial S)$ is embedded in $(C,\partial_+C)$. If $S$ is incompressible and boundary-incompressible, then $S$ is a union of disks.
\end{lemma}
\begin{proof}
 We construct a family $\FF$ of disks and annuli properly embedded in $C$, as follows. First take a defining disk system for $C$. Then, for each component $\Gamma$ of $\partial_-C$:
 \begin{itemize}
  \item if $\Gamma$ is closed, choose a collection of simple closed curves $(\alpha_i,\beta_i)_i$ on $\Gamma$ which are pairwise disjoint, except that for all $i$, $\alpha_i$ intersects $\beta_i$ transversely once, and which cut $\Gamma$ into a disk; take spanning annuli $A_i\cong\alpha_i\times I$ and $B_i\cong\beta_i\times I$,
  \item if $\Gamma$ has a non-empty boundary, choose a collection of properly embedded arcs $(\delta_i)_i$ which cut $\Gamma$ into a disk; take spanning disks $D_i\cong\delta_i\times I$.
 \end{itemize}
 We assume that all these disks and annuli are pairwise disjoint, except each pair $(A_i,B_i)$ which intersects along an arc $\lambda_i$. Note that cutting $C$ along all the disks and annuli in $\FF$ gives a $3$--ball. Slightly isotoping~$S$ if necessary, we also assume that it is transverse to all surfaces in $\FF$. 
 
 Consider the intersection of $S$ with an annulus $A_i$. Note that, since $\partial S\subset\partial_+C$, this intersection cannot contain essential arcs. Further, isotoping~$S$ in a regular neighborhood of $A_i$, we can ensure that $S\cap\lambda_i$ is composed only of one point for each essential curve in $S\cap A_i$. Now we perform a similar isotopy of $S$ in a regular neighborhood of $B_i$, so that $S\cap\lambda_i$ is composed only of one point for each essential curve in $S\cap B_i$. We can further require that this isotopy keeps fixed the points of $S\cap\lambda_i$ that remain at the end of the isotopy, so that finally, each point in $S\cap\lambda_i$ is the intersection of an essential curve~$\xi$ in $A_i$ and an essential curve $\zeta$ in $B_i$.  But then a tubular neighborhood of $\xi\cup\zeta$ in $S$ is a punctured torus. This is impossible since $S$ is a genus--$0$ surface. Hence we can assume that $S$ is disjoint from the arcs $\lambda_i$.
 
 Suppose $S$ meets some $F\in\FF$ along a simple closed curve that bounds a disk in $F$. Since $S$ is incompressible and $C$ is irreducible, a standard innermost disk argument shows that $S$ can be isotoped in order to remove such intersections.  
 
 If $S$ meets an element $F_0$ of $\FF$ along an arc, we can choose one that is outermost in $F_0$. Since $S$ is $\partial$--incompressible, this arc cuts out a disk from $S$. Hence there are arcs of $\FF\cap S$ that are boundary-parallel in $S$. Choose such an arc $\eta$ that is outermost in $S$. Let $\Delta\subset S$ be the associated disk (with $\eta\subset\partial\Delta$) and let $F$ be the element of $\FF$ such that $\eta\subset F\cap S$. Surgering $F$ along $\Delta$ produces two properly embedded surfaces $F_1$ and $F_2$ in $C$ (one disk and one disk or annulus). Cutting $C$ along the elements in $\FF\setminus\{F\}\cup\{F_1,F_2\}$ produces the same result as cutting along $\FF\cup\{\Delta\}$, that is a union of two $3$--balls. One of $F_1$ and $F_2$, say $F_1$, is a disk whose copies on the boundary of these $3$--balls appear on the two different components. Hence replacing $F$ by $F_2$ in $\FF$ preserves the properties of $\FF$ while reducing the number of intersection arcs with $S$. 
 
 Finally, we can assume that $S$ is disjoint from all $\FF$. Cutting $C$ along all disks and annuli in $\FF$ gives a $3$--ball $B$. Take a boundary component $\sigma$ of $S$ which innermost in $\partial B$. The curve $\sigma$ bounds a disk in~$B$, so, since $S$ is incompressible, this component of $S$ is a disk. Iterating the argument, we see that $S$ is a disjoint union of disks.
\end{proof}

\begin{lemma}[\textup{\cite[Lemma II.8]{Jaco}}] \label{lemma:Haken2}
 Let $S$ be a genus--$0$ compact surface which is not a union of disks. Let $\alpha_1,\dots,\alpha_n$ be disjoint properly embedded arcs in $S$ such that:
 \begin{itemize}
  \item $S\ca\cup_{i=1}^n\alpha_i$ is a union of disks,
  \item for all $i$, $\alpha_i$ is essential in $S\ca\cup_{j=i+1}^n\alpha_j$.
 \end{itemize}
 Then $S\ca\cup_{i=1}^n\alpha_i$ has strictly fewer components than $\partial S$.
\end{lemma}

\begin{proof}[Proof of Theorem~\ref{th:Haken}]
 For $\varepsilon=1,2$, set $\Se=S\cap \Ce$. Note that, if both $S_1$ and $S_2$ are unions of disks, then they are connected and $S$ is Haken. Assume $\Se$ is not a union of disks. Perform on $\Se$ as many compressions as possible. At each compression, $\Se$ is surgered into two spheres, at least one of which is essential ({\em resp} surviving); keep that one. Note that the number of components of $\partial \Se$ stricly decreases at each step, so that we end up with an incompressible surface. 
 
 Now assume that $\Se$ is still not a union of disks. We want to perform on $\Se$ a $\partial$--compression, namely isotope $S$ along a $\partial$--compression disk for $S_\varepsilon$. Let $\DD=\sqcup_iD_i$ be a complete collection of spanning disks for $\Ce$ (namely associated to a family of arcs on $\partial_-\Ce$ that cuts each non-closed component of $\partial_-\Ce$ into a disk). By a standard innermost disk argument, we can isotope $\Se$ so that it doesn't meet $\DD$ along a simple closed curve. Also isotope $\Se$ so as to minimize the number of intersecting arcs with $\DD$. If $\Se$ still meets $\DD$, then we can choose an outermost arc $\gamma\subset\Se\cap\DD$ such that there is a subdisk $\Delta$ of $\DD$ whose boundary is the union of $\gamma$ and an arc $\delta\subset\partial_+\Ce$. We perform a $\partial$--compression of $\Se$ along $\Delta$. If $\Se$ is disjoint from $\DD$ (and still $\partial$--compressible), then there is a compression disk $\Delta$ whose boundary is the union of an arc $\gamma\subset\Se$ and an arc $\delta\subset\partial\Ce$. To $\partial$--compress, we need $\delta$ to be disjoint from $\partial_-\Ce$. First, isotoping $\Delta$ in a neighborhood of $\partial_-\Ce$, we can assume that $\delta\cap\partial_-\Ce$ is a disjoint union of arcs parallel to the $\partial D_i\cap\partial_-\Ce$. Then, using a standard innermost disk argument, we isotope $\Delta$ so that $\Delta\cap\DD$ contains no closed component. Further, using an outermost arc argument, we $\partial$--compress $\Delta$ along subdisks of the $D_i$ that meet $\partial\Ce$ along $\partial_+\Ce$ until we get $\Delta\cap\DD=\emptyset$. Finally, for each arc parallel to $\partial D_i\cap\partial_-\Ce$ in $\delta$, we slide $\Delta$ along the disk $D_i$. We end up with $\delta$ contained in $\partial_+\Ce$, and we perform the $\partial$--compression of $\Se$ along $\Delta$.
 
 To see that we end up with $\Se$ being $\partial$--incompressible, we need to prove that the process of $\partial$--compressions terminates. For a compact surface $F$, set $c(F)=2(|\partial F|-|F|)-n(F)$, where $|\cdot|$ stands for the number of connected components and $n(F)$ is the number of components of $F$ other than disks. At each $\partial$--compression, the value of $c(\Se)$ stricly decreases, see Figure~\ref{fig:delcompression}.
 
 We finally end up with $\Se$ being incompressible and $\partial$--incompressible. Thanks to Lemma~\ref{lemma:Haken1}, $\Se$ is now a union of disks, and by Lemma~\ref{lemma:Haken2}, the number of components of $\partial \Se$ has strictly decreased. 
 Iterate this process until $S\cap\Sigma$ is connected.
\begin{figure}[htb]
\begin{center}
\begin{tikzpicture}
 \foreach \y in {0,3} {
 \begin{scope} [yshift=\y cm]
 \foreach \x in {0,6} {
 \draw[fill=blue!30] (\x,0) ellipse (2 and 1);
 \draw[fill=white,xshift=\x cm] (-1,0) circle (0.3);}
 \draw[fill=white] (0,0.2) circle (0.3) (1.1,-0.2) circle (0.3);
 \draw[->,>=latex,thick] (2.6,0) -- (3.4,0);
 \end{scope}}
 \draw[red] (0.3,0.1) -- (0.8,-0.1) node[pos=0.5,above] {$\gamma$};
 \draw[fill=white,rounded corners=2pt] (5.7,0.2) arc (180:320:0.3) -- (6.8,-0.2) arc (-180:140:0.3) -- (6.3,0.2) arc (0:180:0.3);
 \draw[red] (0.9,3.02) .. controls +(-1.2,1.8) and +(-2.5,-0.2) .. (0.8,2.7) (0.8,3.5) node {$\gamma$};
 \draw[fill=white,rounded corners=2pt] (6.8,3.2) .. controls +(-1,1.5) and +(-2.5,-0.2) .. (6.8,2.7) arc (-160:120:0.3) -- (6.8,3.2);
 \draw[fill=blue!30,rounded corners=1pt]  (6.85,3) .. controls +(-1.1,1.6) and +(-2.2,-0.3) .. (6.8,2.8) -- (6.85,3);
 \draw[fill=white] (6,3.2) circle (0.3);
\end{tikzpicture} \caption{Effect of a $\partial$--compression on $\Se$} \label{fig:delcompression}
\end{center}
\end{figure}
\end{proof}

Haken's theorem can be adapted to properly embedded disks. We say that an essential properly embedded disk in a $3$--manifold with a Heegaard splitting is {\em Haken} if it intersects the Heegaard surface transversely along a single arc.

\begin{theorem} \label{th:HakenDisk}
 Let $M$ be a sutured $3$--manifold and $M=C_1\cup_\Sigma C_2$ a Heegaard splitting. If there is an essential properly embedded disk $D\subset M$ whose boundary meets each of $\partial_-C_1$ and $\partial_-C_2$ along an arc, then there is an essential properly embedded disk $D'\subset M$ which is Haken and such that $\partial D'=\partial D$.
\end{theorem}
\begin{proof}
 One can reproduce the proof of Theorem~\ref{th:Haken}. The first difference appears in the compression process. When performing a compression on $\De=D\cap\Ce$, we get a disk and a sphere: we keep the disk. The main point to take care of concerns the $\partial$--compression process. We need to choose the spanning disk set $\DD$ so that one of the disks, say $D_0$, meets $\partial_-\Ce$ along an arc parallel to $D\cap\partial_-\Ce$. If the component $\De^-$ of $\De$ which meets $\partial_-\Ce$ is a disk, we even choose $D_0$ to be parallel to $\De^-$. Otherwise, after having made $\De$ disjoint from $\DD$, we obtain a $\partial$--compression disk by taking a disk parallel to $D_0$ and slightly sliding it on a neighborhood of $\De^-\cap\partial_-\Ce$ in $\De^-$. After performing this $\partial$--compression, the new component $\De^-$ is parallel to $D_0$, and the rest of the proof works as well.
 
 If the obtained disk $D'$ is essential, we are done. Otherwise, $\partial D$ bounds a disk $D_\partial$ in $\partial M$. In this case, $D\cup D_\partial$ is an essential sphere in $M$. By Theorem~\ref{th:Haken}, there is a Haken sphere $S$ in $M$. Then we can produce a Haken disk by tubing $D_\partial$ and $S$ along a tube around an arc in $\Sigma$.
\end{proof}

\subsection{Strong Haken's theorem for spheres and disks in sutured Heegaard splittings}

We now prove a strong Haken's theorem for sutured Heegaard splittings, following Hensel and Schultens~\cite{HS}. We need a condition on the splitting. We say that a sutured decomposition on the boundary of a $3$--manifold $M$ is {\em admissible} if every $2$--sphere in $\partial M$ contains exactly one suture. In this case, we will also use the expressions {\em admissible sutured $3$--manifold}, {\em admissible Heegaard splitting}.

\begin{theorem} \label{th:StrongHaken}
 Let $M$ be an admissible sutured $3$--manifold and $M=C_1\cup_\Sigma C_2$ a Heegaard splitting. Every essential embedded $2$--sphere in $M$ is isotopic to a Haken sphere.
\end{theorem}

We need two lemmas, which are the analogous to \cite[Lemma 3.7 $\&$ Lemma 3.5]{HS} in the setting of sutured Heegaard splittings. We reproduce the proofs for the sake of completeness.

\begin{lemma} \label{lemma:non-surviving}
 In an admissible sutured Heegaard splitting, every non-surviving sphere is isotopic to a Haken sphere.
\end{lemma}
\begin{proof}
 Let $S$ be a non-surviving sphere in $M$. Then $S$ bounds an embedded punctured $3$--ball $B\subset M$, with $k\geq0$ punctures. If $k=0$, then $S$ bounds an embedded $3$--ball in $M$, so that it can be isotoped into the neighborhood of a point on the Heegaard surface $\Sigma$ so as to intersect it in a single simple closed curve. If $k=1$, then $S$ is isotopic to a boundary component, so the admissibility condition implies that it is isotopic to a Haken sphere. 
 
 Assume $k=2$, and denote by $\sigma_1$ and $\sigma_2$ the two boundary components of $M$ contained in $B$. Then $S$ is isotopic to the boundary of a regular neighborhood of $\sigma_1\cup\sigma_2\cup\alpha$, where $\alpha$ is a properly embedded arc in $M$ joining $\sigma_1$ and $\sigma_2$. The arc $\alpha$ can be isotoped in a regular neighborhood of the Heegaard surface $\Sigma$, of the form $\Sigma\times I$, so that the projection to $\Sigma$ has only transverse double points as singularities. In order to isotope $\alpha$ into $\Sigma$, we shall prove that we can remove these crossings. For this, follow the arc $\alpha$ from its endpoint on $\sigma_1$ until you meet a crossing. Then isotope the transverse strand following the beginning of $\alpha$, and over one disk of $\sigma_1$ on either side of $\Sigma$, see Figure~\ref{fig:removecross}. Iterating this process, we see that we can isotope $\alpha$ into~$\Sigma$. At this point, $S$ intersects $\Sigma$ in a single curve.
 
\begin{figure}[htb] 
\begin{center}
\begin{tikzpicture} [scale=0.8]
\foreach \x in {0,12} {
\begin{scope} [xshift=\x cm]
 \draw[fill=gray!30] (0,0) node {$\sigma_1$} circle (1);
 \draw[thick] (1,0) -- (5,0);
 \draw[thick,dashed] (5,0) -- (6,0) (3.5,1) -- (3.5,1.7) (3.5,-1.7) -- (3.5,-1);
\end{scope}}
 \draw[thick] (3.5,1) -- (3.5,0.2) (3.5,-0.2) -- (3.5,-1);
 \draw[thick,rounded corners=3pt] (15.5,1) -- (15.5,0.2) -- (13.2,0.2) arc (10:350:1.2) -- (15.5,-0.2) -- (15.5,-1);
 \draw[line width=2pt,->] (8,0) -- (9,0);
\end{tikzpicture}
\caption{Removing the crossings on the arc} \label{fig:removecross}
\end{center}
\end{figure}

 Now assume that $k>2$. Then $B$ contains a $2$--sphere $S'$ that bounds a twice-punctured $3$--ball $B'\subset B$. By the previous case, $S'$ can be isotoped to be a Haken sphere. Then the manifold $M'=M\setminus\Int(B')$ inherits an admissible Heegaard splitting from that of $M$. By induction on $k$, $S$ is isotopic to a Haken sphere in $M'$, thus also in $M$.
\end{proof}

The proof of Theorem~\ref{th:StrongHaken} involves the graph of surviving spheres $\Ss(M)$, defined as follows.
\begin{itemize}
 \item The vertices are the isotopy classes of surviving spheres in $M$.
 \item There is an edge between two vertices if the corresponding isotopy classes can be realized by disjoint spheres.
\end{itemize}

\begin{lemma} \label{lemma:graph}
 The graph $\Ss(M)$ is connected.
\end{lemma}
\begin{proof}
 Assume there are non-isotopic surviving spheres $S$ and $S'$ in $M$. Assume they minimize their number of intersection curves within their isotopy classes. If they do intersect, choose an intersection curve $c$ which is innermost in $S$, thus bounds a disk $\Delta\subset S$ whose interior is disjoint from $S'$. Then surger $S'$ along $\Delta$: this provides two embedded spheres in $M$, disjoint from $S'$ and having fewer intersection curves with $S$, at least one of which is surviving. Iterating, we get a path between the isotopy classes of $S$ and $S'$ in $\Ss(M)$.
\end{proof}

\begin{proof}[Proof of Theorem~\ref{th:StrongHaken}]
 We proceed by induction on $(g,n)=(g(\Sigma),|\partial\Sigma|)$ with lexicographic order.
 
 If $(g,n)\preceq(0,2)$, then $M$ is a $3$--sphere if $n=0$, a $3$--ball if $n=1$, and a solid torus or a product of a $2$--sphere and an interval if $n=2$. Hence $M$ contains no surviving sphere and Lemma~\ref{lemma:non-surviving} concludes. Now fix $(g,n)\succ(0,2)$.
 
 We first prove the following claim: 
 if $S\subset M$ is a surviving Haken sphere and $T\subset M$ is a surviving sphere disjoint from $S$, then $T$ is isotopic to a Haken sphere. To see this, let $N$ be the component of $M\ca S$ (over $1$ or $2$ components) which contains $T$. If the curve $\gamma=S\cap\Sigma$ is non-separating on $\Sigma$, then $g(\Sigma\ca\gamma)<g(\Sigma)$. If $\gamma$ is separating, note that $\gamma$ cannot bound a disk or be boundary-parallel on $\Sigma$. Indeed, it would imply that $S$ is a boundary component of a manifold with Heegaard surface a disk or an annulus. Such manifolds are listed above, and we see that it contradicts the fact that $S$ is surviving. It implies that the Heegaard splitting induced on  $N$ by that of $M$ has either a strictly lower genus, or a Heegaard surface with strictly less boundary components. Hence we can apply the inductive assumption.
 
 Now, the result is given by Lemma~\ref{lemma:non-surviving} if $M$ contains no surviving sphere. Otherwise, $M$ contains a surviving Haken sphere $S_0$ by Theorem~\ref{th:Haken}. If $S$ is an essential $2$--sphere in $M$ non-isotopic to $S_0$, either it is non-surviving and we apply Lemma~\ref{lemma:non-surviving}, or it is surviving and there is a path of surviving spheres from $S_0$ to $S$ in which successive spheres are disjoint, then we apply the above claim.
\end{proof}

The strong Haken theorem can also be adapted to properly embedded disks. For this, we need a stronger condition on the Heegaard splitting. We say that a sutured decomposition on the boundary of a $3$--manifold $M$ is {\em strongly admissible} if it is admissible and every disk in $\partial_0M$ or $\partial_1M$ is contained in an $S^2$--component of $\partial M$. In this case, we will also use the expressions {\em strongly admissible sutured $3$--manifold}, {\em strongly admissible Heegaard splitting}.

We also need a condition on the boundary of the embedded disk. Given a simple closed curve $\gamma$ in the boundary of a sutured $3$--manifold $M$, we say that $\gamma$ is {\em balanced} if $\gamma$ intersects $\partial_0M$ and $\partial_1M$ both along essential arcs or both along inessential arcs.

\begin{theorem} \label{th:StrongHakenDisk}
 Let $M$ be a strongly admissible sutured $3$--manifold and $M=C_1\cup_\Sigma C_2$ a Heegaard splitting. Every essential properly embedded disk $D\subset M$ with balanced boundary is isotopic to a Haken disk.
\end{theorem}

\begin{lemma} \label{lemma:nonessentialdisk}
 With the conditions of Theorem~\ref{th:StrongHakenDisk}, if $\partial D$ bounds a disk in $\partial M$, then $D$ is isotopic to a Haken disk.
\end{lemma}
\begin{proof}
 Let $\Delta\subset\partial M$ be a disk such that $\partial\Delta=\partial D$. Note that the strong admissibility condition implies that $\Delta$ is made of one disk in each of $\partial_-C_1$ and $\partial_-C_2$. 
 There exists an embedded $2$--sphere $S\subset M$ such that $D$ is isotopic to $\Delta\cup S$ tubed along an arc $\eta$, see Figure~\ref{figDnonessential}. By Theorem~\ref{th:StrongHaken}, the sphere $S$ can be  isotoped to a Haken sphere. Then the arc $\eta$ can be isotoped into a neighborhood of the Heegaard surface~$\Sigma$. Like in the proof of Lemma~\ref{lemma:non-surviving}, $\eta$ can be further isotoped to lie in $\Sigma$. Hence $D$ is isotopic to a Haken disk. 
\end{proof}

\begin{figure}[htb] 
\begin{center}
\begin{tikzpicture} [xscale=0.5,yscale=0.4]
 \draw (3,0) arc (260:100:4) node[above left] {$\partial M$};
 \draw (-0.8,3.5) node {$\Delta$};
 \draw[thick,rounded corners=4pt] (1,0.95) node {$\bullet$} -- (3,2) -- (1.5,2.5) -- (2,4) node[right] {$D$} -- (1.2,5) -- (2.5,5.7) -- (1,6.85) node {$\bullet$};
 \begin{scope} [xshift=8cm]
  \draw (3,0) arc (260:100:4) node[above left] {$\partial M$};
 \draw (-0.8,3.5) node {$\Delta$};
 \draw[thick,rounded corners=4pt] (1,0.95) node {$\bullet$} arc (220:180:4.3) -- (0.6,3.7)  arc (180:216:4.3) -- (3,2) -- (1.5,2.5) -- (2,4) -- (1.2,5) -- (2.5,5.7) -- (2,6.1)
 (1,6.85) node {$\bullet$} arc (140:180:4.3) -- (0.6,4.05) arc (180:144:4.3) -- (2,6.1);
 \end{scope}
 \begin{scope} [xshift=16cm]
  \draw (3,0) arc (260:100:4) node[above left] {$\partial M$};
 \draw (-0.8,3.5) node {$\Delta$};
 \draw[thick,rounded corners=4pt] (1,0.95) node {$\bullet$} arc (228:132:4) node {$\bullet$}
 (0.6,3.7)  arc (180:216:4.3) -- (3,2) -- (1.5,2.5) -- (2,4) node[right] {$S$} -- (1.2,5) -- (2.5,5.7) -- (1.4,6.6) arc (144:185:4.3);
 \draw (-0.3,3.8) -- (0.6,3.8) (0.2,3.8) node[above] {$\eta$};
 \end{scope}
\end{tikzpicture}
\caption{Isotoping the disk $D$ onto the tubing of $\Delta$ and a sphere $S$} \label{figDnonessential}
\end{center}
\end{figure}

Given an essential simple closed curve $\gamma\subset\partial M$, we define the graph of essential disks $\DD_\gamma(M)$ as follows: the vertices are the isotopy classes of essential properly embedded disks in $M$ with boundary parallel to $\gamma$, and there is an edge between two vertices if the corresponding isotopy classes can be realized by disjoint disks. 

\begin{lemma} \label{lemma:graphdisk}
 The graph $\DD_\gamma(M)$ is connected.
\end{lemma}
\begin{proof}
 Let $D$ and $D'$ be two non-isotopic disks representing classes in $\DD_\gamma(M)$. Since their boundaries are both parallel to $\gamma$, we can isotope $D'$ in order to have $\partial D$ and $\partial D'$ disjoint. Then we conclude as in the proof of Lemma~\ref{lemma:graph}: when we surger one of the disks, we get a disk and a sphere, and we keep the disk, which is necessarily essential since $\gamma$ is essential.
\end{proof}

\begin{proof}[Proof of Theorem~\ref{th:StrongHakenDisk}]
 If every essential disk in $M$ has a non-essential or non-balanced boundary, then Lemma~\ref{lemma:nonessentialdisk} concludes. Assume $M$ contains an essential disk with essential and balanced boundary $\gamma$. Then, by Theorem~\ref{th:HakenDisk}, there is a disk $D_0$ in $M$ with boundary $\gamma$ which is Haken. If $D$ is an essential disk with boundary parallel to $\gamma$, which is disjoint from and non-isotopic to $D_0$, then $D$ is essential in~$M\ca D_0$, which inherits a strongly admissible Heegaard splitting. Hence Lemma~\ref{lemma:nonessentialdisk} shows that $D$ is isotopic to a Haken disk. Finally we conclude with Lemma~\ref{lemma:graphdisk}.
\end{proof}

\begin{corollary} \label{cor:strongHaken}
 Let $M$ be a strongly admissible sutured $3$--manifold and $M=C_1\cup_\Sigma C_2$ a Heegaard splitting. Then every collection of disjoint essential embedded $2$--spheres and properly embedded disks with balanced boundaries is isotopic to a collection of Haken spheres and disks.
\end{corollary}
\begin{proof}
 Start with a component $F$ of the collection, either a sphere or a disk. By Theorems~\ref{th:StrongHaken} and~\ref{th:StrongHakenDisk}, $F$ can be isotoped to be Haken. Extend the isotopy to an ambiant isotopy and apply it to the whole collection. Then cutting along $F$ yields a new strongly admissible sutured Heegaard splitting, with one surface less in the collection.
\end{proof}

\subsection{Sutured Heegaard splittings of double compression bodies} \label{sec:doublecomp}

We are interested in understanding the Heegaard splittings of {\em double compression bodies}, namely compact $3$--manifolds obtained by gluing two copies of a given compression body along their positive boundaries via the identity map. Such a double compression body can be written as $M=\big(\sharp(P\times I)\big)\sharp\big(\sharp^k (S^1\times S^2)\big)$, where $P$ is a compact surface, connected or not, possibly empty, containing no $2$--sphere, and $\sharp(P\times I)$ is the connected sum of all the components of $(P\times I)$. We define the {\em standard sutured decomposition} of $\partial M$ as follows: $P\times \{0\}\subset\partial_0 M$, $P\times\{1\}\subset\partial_1 M$, and the suture set is $s=\partial P\times\{\frac12\}$. We extend the definition of this standard sutured decomposition to punctured compression bodies, namely compression bodies with the interior of some $3$--balls removed, by adding one suture on each $S^2$--boundary component associated to a puncture. 
Note that this sutured decomposition is strongly admissible.

Thanks to Theorem~\ref{th:StrongHaken}, we essentially need to understand the splittings of products $P\times I$ with $P$ connected, possibly with punctures.
Let $P$ be a compact connected oriented surface of genus $g\geq0$. Let $B_1,\dots,B_k$ be disjoint closed $3$--balls embedded in the interior of $P\times I$, each of which intersects $P\times\{\frac12\}$ transversely along a $2$--disk, where $k\geq0$. Set $\Bb=\cup_{i=1}^kB_i$, $M=(P\times I)\setminus\Int(\Bb)$, and $\Sigma_0=M\cap(P\times\{\frac12\})$. We fix the standard sutured decomposition on $\partial M$, with the suture set $s=\partial\Sigma_0$. 

\begin{proposition} \label{prop:HSpuncproduct}
 The sutured Heegaard splitting of $M$ defined by the Heegaard surface $\Sigma_0$ is the unique genus--$g$ sutured Heegaard splitting of $M$ up to isotopy. It is the minimal genus sutured Heegaard splitting of $M$.
\end{proposition}
\begin{proof}
 Let $M=C_1\cup_\Sigma C_2$ be a sutured Heegaard splitting of $M$. Recall that it implies $\partial\Sigma=s$. Then $\Sigma$ is the positive boundary of $C_1$, whose negative boundary is the union of a copy of $P$ and some $2$--disks. It follows that the genus of $\Sigma$ is at least that of $P$. Hence $\Sigma_0$ defines a minimal genus Heegaard splitting of $M$. We now assume that $g(\Sigma)=g(P)$.
 
 For $i=1,\dots,k$, choose a properly embedded arc $\gamma_i$ in $C_1$ with an end on the disk $D_i=\partial B_i\cap C_1$ and the other end on $P\times\{0\}$ (hence the ends of $\gamma_i$ lie on different components of $\partial_-C_1$). Since $C_1$ is obtained from a thickening of $\partial_-C_1$ by gluing $k$ $1$--handles (which make $C_1$ connected), 
 we can choose the arcs~$\gamma_i$ so that $C_1$ is a regular neighborhood of $\partial_-C_1\cup\big(\cup_{i=1}^k\gamma_i\big)$.  
 Hence the uniqueness of the Heegaard splitting will follow from the uniqueness of the collection of arcs $(\gamma_i)_{1\leq i\leq k}$ up to isotopy. We consider here isotopies within similar collections of arcs, which means that the ends are not fixed, but they must remain in~$\partial_-C_1$. 
 
 Such collections of arcs are homotopic. To see this, we can parametrize $C_1$ as $P\times I$, with the disks $D_i$ contained in $P\times\{1\}$. Using the deformation-retract of $P\times I$ onto $P\times\{0\}$, we see that $\gamma_i$ is homotopic to the union of its projection on $P\times\{0\}$ and of $\{p\}\times I$ for some point $p\in P$. Since we can move the end of the arc lying in $P\times\{0\}$ (remaining in $P\times\{0\}$), it follows that $\gamma_i$ is homotopic to $p\times I$.
 
 Now we need to show that homotopic collections of arcs are isotopic. We shall see that we can allow a strand to cross another. Indeed, any point of an arc $\gamma_i$ is the center of a properly embedded disk in $C_1$ whose boundary is the suture which lies on $\partial B_i$. This suture also bounds a disk $\Delta$ in~$C_2$. Hence we can slide part of any $\gamma_j$ (including~$\gamma_i$) along this disk $\Delta$ in order to realize a crossing with~$\gamma_i$.
\end{proof}

\begin{theorem} \label{th:doublecompHS}
 Consider a double compression body $M=\big(\sharp(P\times I)\big)\sharp\big(\sharp^k (S^1\times S^2)\big)$, where $P$ is a compact surface which contains no $2$--sphere and $k\geq0$, with the standard sutured decomposition. The sutured manifold $M$ admits a unique minimal genus Heegaard splitting up to isotopy, and this minimal genus is $g(P)+k$.
\end{theorem}
\begin{proof}
 If $F$ is a component of $P$, the minimal genus Heegaard splitting of $F\times I$ is given in Proposition~\ref{prop:HSpuncproduct}. The components $S^1\times S^2$ have a standard genus--$1$ splitting. The connected sum of all these splittings gives a Heegaard splitting of $M$ of genus $g(P)+k$.
 
 Now let $M=C_1\cup_\Sigma C_2$ be any Heegaard splitting of $M$, of some genus $g$.
 
 Thanks to Theorem~\ref{th:StrongHaken}, there are $k$ non-separating Haken spheres and some number of separating Haken spheres, all disjoint, such that cutting along all these spheres yields a disjoint union of punctured products $F\times I$ with $F$ a connected surface. We obtain on each punctured $F\times I$ a Heegaard splitting of genus at least $g(F)$. Since cutting along a non-separating sphere reduces the genus of the Heegaard surface by one, we see that the initial Heegaard splitting had genus at least $g(P)+k$.  
 
 Now assume that $\Sigma$ and $\Sigma'$ are two Heegaard surfaces for $M$, both of genus $g(P)+k$. We choose as before a collection $\Ss$ of disjoint $2$--spheres in $M$, $k$ of which are non-separating, such that cutting along these spheres yields punctured products $F\times I$. When applying Theorem~\ref{th:StrongHaken} to put these spheres in Haken position with respect to $\Sigma$, we can isotope the surface $\Sigma$ instead of isotoping the spheres. Then we similarly isotope $\Sigma'$, so that we end up with the spheres of $\Ss$ being Haken with respect to both $\Sigma$ and~$\Sigma'$. Since every two oriented simple closed curves on a $2$--sphere are isotopic, we can further isotope $\Sigma'$ to obtain that $\Sigma$ and $\Sigma'$ meet each sphere in $\Ss$ along the same curve, with matching orientations. Now, when cutting along the spheres in $\Ss$, $\Sigma$ and $\Sigma'$ define Heegaard splittings on each punctured $F\times I$ with the same sutured decomposition, which is the standard one. By Proposition~\ref{prop:HSpuncproduct}, the Heegaard surfaces induced by $\Sigma$ and $\Sigma'$ in each punctured $F\times I$ are isotopic, so that $\Sigma$ and $\Sigma'$ are isotopic in $M$.
\end{proof}

In \cite{Wald}, Waldhausen classified the Heegaard splittings of the $3$--sphere. 
Together with the strong Haken's theorem and the above result, we shall see that it provides a complete classification of Heegaard splittings of double compression bodies with the standard sutured decomposition.

\begin{theorem}[Waldhausen] \label{th:Wald}
 Every Heegaard splitting of $S^3$ of positive genus is stabilized.
\end{theorem}

This result and the Haken theorem provide a classification of Heegaard splittings of $\sharp^g(S^1\times S^2)$ up to diffeomorphisms. As explained in \cite{HS}, one can get a classification up to isotopy, using the strong Haken theorem and the classification of genus--$0$ splittings of a punctured $3$--ball (\cite[Theorem~3.3]{HS}, also a special case of Proposition~\ref{prop:HSpuncproduct}). This is recovered in Theorem~\ref{th:doublecompHSfull} below (with $P$ empty). We start with a corollary of Waldhausen's result, which is part of \cite[Theorem~3.3]{HS}.

\begin{corollary} \label{cor:Wald}
 Every admissible Heegaard splitting of a punctured $S^3$ of positive genus is stabilized.
\end{corollary}
\begin{proof}
 We fill in each boundary component with a $3$--ball containing a properly embedded $2$--disk which caps off the Heegaard surface. We obtain a Heegaard splitting of $S^3$, which has a positive genus. By Theorem~\ref{th:Wald}, this splitting is stabilized, so that our initial splitting was already stabilized.
\end{proof}

The next result is due to Scharlemann and Thompson \cite{ST} in the case of $P$ closed, with no puncture. We recover and generalize their result (in the case of the standard sutured decomposition).
\begin{theorem} \label{th:ST}
 Let $M$ be a punctured $P\times I$, where $P$ is a compact connected oriented surface of genus~$g$, with the standard sutured decomposition on $\partial M$. Every Heegaard splitting of genus $n>g$ of $M$ is stabilized.
\end{theorem}
\begin{proof}
 Fix a Heegaard splitting $M=C_1\cup_\Sigma C_2$. 
 First assume that $\partial P$ has a non-empty boundary. Then there is a collection of disjoint spanning disks in $M$ associated to a collection of arcs in $P$ that cut $P$ into a disk. Thanks to Corollary~\ref{cor:strongHaken}, we can isotope these disks into Haken disks. Now cutting along these disks yields a punctured $S^3$ with an admissible Heegaard splitting of positive genus. By Corollary~\ref{cor:Wald}, this splitting is stabilized, so that the initial one also was.
 
 Now assume that $P$ is closed. We can choose a solid tube $D^2\times I\subset P\times I$ which intersects each of $\Sigma$, $\partial_-C_1$, and $\partial_-C_2$ along a disk. Cutting out the interior of this tube leads us to the previous case.
\end{proof}

\begin{theorem} \label{th:doublecompHSfull}
 Consider a double compression body $M=\big(\sharp(P\times I)\big)\sharp\big(\sharp^k (S^1\times S^2)\big)$, where $P$ is a compact surface which contains no $2$--sphere and $k\geq0$, with the standard sutured decomposition of $\partial M$. For all $n\geq g(P)+k$, the sutured manifold $M$ admits a unique Heegaard splitting of genus $n$ up to isotopy.
\end{theorem}
\begin{proof}
 Assume $n>g(P)+k$. We need to prove that, in this case, the Heegaard splitting is stabilized. 
 Thanks to Theorem~\ref{th:StrongHaken}, there are $k$ non-separating Haken spheres and some number of separating Haken spheres, all disjoint, such that cutting along all these spheres yields a disjoint union of punctured products $F\times I$ with $F$ a connected surface. We obtain a Heegaard splitting on each punctured $F\times I$, one of which has a genus greater than $g(F)$. By Theorem~\ref{th:ST}, this splitting is stabilized, so that our initial splitting of $M$ was already stabilized.
\end{proof}

\section{Diffeomorphisms of $4$--dimensional compression bodies and relative trisections}
\label{secLPrel}

\subsection{A relative Laudenbach--Poénaru's theorem}

In this section, we adapt our proof of Laudenbach--Poénaru's theorem to a relative setting. We had two preliminaries for this proof. The first one is the classification of minimal genus splittings of double handlebodies. We generalized it in Theorem~\ref{th:doublecompHS} to double compression bodies. The second one is the fact that every diffeomorphism of a $3$--dimensional handlebody which is the identity on the boundary is isotopic to the identity. We now generalize it to compression bodies.

\begin{lemma}
 Let $C$ be a compression body. Every two defining disk systems for $C$ which coincide on the boundary are isotopic.
\end{lemma}
\begin{proof}
 We have proved in Lemma~\ref{lemma:compirr} that compression bodies are irreducible. Hence a standard innermost disk argument concludes.
\end{proof}

\begin{lemma} \label{lemma:diffeocompbody}
 Let $C$ be a compression body. Let $\varphi$ be a diffeomorphism of $C$. If $\varphi$ is the identity on $\partial_+ C$, then $\varphi$ is isotopic to the identity, relative to the positive boundary.
\end{lemma}
\begin{proof}
 Pick a defining disk system $\DD$ for $C$. By the same reasoning as in the proof of Lemma~\ref{lemma:diffeohandlebody}, we can assume that $\varphi$ is the identity on $\partial_+C\cup\DD$. We are led to a diffeomorphism of a product $\partial_-C\times I$ which is the identity on $\partial_-C\times\{1\}$. Interpolating with the identity provides the required isotopy.
\end{proof}

\begin{definition}
 A \emph{(lensed) hyper compression body} $V$ is a smooth connected manifold constructed as follows:
 \begin{itemize}
  \item start with $M\times I$ where $M$ is a compact oriented $3$--manifold,
  \item glue $4$--dimensional $1$--handles along $M\times\{1\}$,
  \item collapse the vertical boundary along the $I$--factor.
 \end{itemize}
 The {\em negative boundary} $\partial_-V$ is defined as $M\times\{0\}$ and the {\em positive boundary} is $M\times\{1\}$, so that $\partial V=\partial_-V\cup_\partial\partial_+V$. 
\end{definition}

We will consider hyper compression bodies with a specific condition on the negative boundary. We say that a hyper compression body $V$ is {\em $P$--based} if $\partial_-V$ is a trivial compression body $P\times I$, where $P$ is a compact oriented surface. Further, the positive boundary is diffeomorphic to a double compression body $\big(\sharp(P\times I)\big)\sharp\big(\sharp^k (S^1\times S^2)\big)$. On the common boundary of $\partial_+V$ and $\partial_-V$, we consider the standard sutured decomposition.

\begin{theorem} \label{th:LPrel}
 Let $V$ be a $P$--based hyper compression body. Assume that $P$ contains no $2$--sphere. Then every diffeomorphism of $\partial_+V$ that preserves the standard sutured decomposition on $\partial(\partial_+V)$ extends to a diffeomorphism of $V$.
\end{theorem}
\begin{proof}
 Fix an identification of $V$ with a product $C\times I$, with $C$ a $3$--dimensional compression body, where $\partial_+ C\times I$ has been collapsed along the $I$--factor. This induces a foliation of $V$ by compression bodies~$C_t$, $t\in[0,1]$, which intersect along their positive boundary $\Sigma=\partial_+C_t$, such that $\partial_-C_0=\partial_0(\partial_+V)$ and $\partial_-C_1=\partial_1(\partial_+V)$. This surface $\Sigma$ defines a minimal genus Heegaard splitting of $\partial_+V$.
 
 Now take a diffeomorphism $\varphi$ of $\partial_+V$. It sends the sutured Heegaard splitting $\partial_+V=C_0\cup_\Sigma C_1$ onto another splitting $\partial_+V=\varphi(C_0)\cup_{\varphi(\Sigma)} \varphi(C_1)$, for the same sutured decomposition, and with the same genus. By Theorem~\ref{th:doublecompHS}, they are isotopic. Like in the proof of Theorem~\ref{th:LP}, realizing the isotopy in a collar neighborhood of $\partial_+V$, we can assume that $\varphi(\Sigma)=\Sigma$ and $\varphi(C_t)=C_t$ for each $t=0,1$. 
 
 Now $\varphi_{|C_0}$ and $\varphi_{|C_1}$ define two diffeomorphisms of $C$ which coincide on the positive boundary. By Lemma~\ref{lemma:diffeocompbody}, there is an isotopy $\varphi_t$ of diffeomorphisms of $C$ from $\varphi_{|C_0}$ to $\varphi_{|C_1}$. The map $\phi:V\to V$ induced by the diffeomorphism $(x,t)\mapsto (\varphi_t(x),t)$ of $C\times I$ is the desired diffeomorphism.
\end{proof}

\subsection{Diagrams of $4$--dimensional multisections}

In this section, we apply Theorem~\ref{th:LPrel} to show that a multisection diagram determines a diffeomorphism class of $4$--manifolds. We work in the setting of relative multisections in the sense of Islambouli--Naylor \cite{IN}, a generalization of Gay and Kirby's trisections~\cite{GayKirby}.

\begin{definition}\label{def:Multisection}
 A \emph{multisection} of a compact oriented $4$--manifold $X$ is a decomposition $X=\cup_{i=1}^nX_i$ with the following properties (all arithmetic involving indices is mod $n$):
 \begin{enumerate}
  \item $\displaystyle\Sigma=\bigcap_{i=1}^n X_i$ is a compact connected oriented surface,
  \item when $|i-j|>1$, $X_i\cap X_j=\Sigma$.
  \item $C_i=X_i\cap X_{i+1}$ is a $3$--dimensional compression body satisfying $\partial_+C_i=\Sigma$ and $\partial_-C_i=C_i\cap \partial X$,
  \item there is a compact oriented surface $P$, which contains no $S^2$, such that each $X_i$ is a $P$--based hyper compression body with $\partial_+X_i=C_{i-1}\cup_\Sigma C_i$ and $\partial_-X_i=X_i\cap \partial X$, and the natural sutured decomposition of $\partial(\partial_-X_i)$ coincides with the decomposition $\partial(\partial_-X_i)=\partial_-C_{i-1}\cup\partial_-C_i$,
  \item the surface $\Sigma$ is smoothly properly embedded in $X$, the $C_i$ are submanifolds with corners whose codimension--$2$ stratum is $\partial\Sigma$, the $X_i$ are submanifolds with corners, whose codimension--$2$ stratum is $\Sigma\cup\partial_-C_{i-1}\cup\partial_-C_i$ and whose codimension--$3$ stratum is $\partial\Sigma$.
 \end{enumerate}
 A multisection is called a \emph{trisection} when $n=3$.
\end{definition}

\begin{figure}[htb]
\begin{center}
\begin{tikzpicture} [scale=0.25]
 \draw (0,0) circle (5);
 \foreach \s in{1,...,5,6} {
 \draw[rotate=60*(\s+1)] (0,0) -- (5,0);
 \draw[rotate=60*(\s+1)] (6,0) node {$C_\s$};
 \draw[rotate=60*\s+30] (3,0) node {$X_\s$};}
 \draw (0,0) node {$\scriptstyle\bullet$} (0.2,-0.2) node[above right] {$\scriptstyle\Sigma$};
\end{tikzpicture}
\end{center}
\caption{Schematic of a multisection}
\label{fig:multisection}
\end{figure}

A {\em diagram} of such a multisection is a tuple $(\Sigma;\alpha_1,\dots,\alpha_n)$ where $\Sigma$ is the central surface of the multisection and $\alpha_i$ is a cut-system for $C_i$. Note that the positive boundaries of the $X_i$ are double compression bodies, so that Theorem~\ref{th:doublecompHSfull} implies that each subdiagram $(\Sigma;\alpha_{i-1},\alpha_i)$ is handleslide diffeomorphic to a diagram as represented in Figure~\ref{fig:StandardDiagram}. 
Note that any abstract diagram satisfying this property is a diagram of some multisected $4$--manifold. When the surface $P$ has no closed component, Castro, Gay and Pinz\'on-Caicedo proved that a multisection diagram determines a unique $4$--manifold up to isotopy \cite{CGPC2}. Theorem~\ref{th:LPrel} allows us to recover this fact and to extend it to any surface $P$ containing no $S^2$--component. We will see below that this is optimal.

\begin{figure}[htb]
\begin{center}
\begin{tikzpicture} [xscale=0.3,yscale=0.28]
\newcommand{\trou}{
(2,0) ..controls +(0.5,-0.25) and +(-0.5,-0.25) .. (4,0)
(2.3,-0.1) ..controls +(0.6,0.2) and +(-0.6,0.2) .. (3.7,-0.1)}
\draw (0,0) ..controls +(0,1.2) and +(-2,1) .. (4,2) ..controls +(1,-0.5) and +(-1,0) .. (6,1.25);
\draw (0,0) ..controls +(0,-1.2) and +(-2,-1) .. (4,-2) ..controls +(1,0.5) and +(-1,0) .. (6,-1.25);
\foreach \x/\y in {6/0,18/0} {
\begin{scope} [xshift=\x cm,yshift=\y cm]
\draw (0,1.25) ..controls +(1,0) and +(-2,1) .. (4,2) ..controls +(2,-1) and +(-2,-1) .. (8,2) ..controls +(2,1) and +(-1.2,0) .. (12,1.25);
\draw (0,-1.25) ..controls +(1,0) and +(-2,-1) .. (4,-2) ..controls +(2,1) and +(-2,1) .. (8,-2) ..controls +(2,-1) and +(-1.2,0) .. (12,-1.25);
\end{scope}}
\foreach \x in {0,6,12,18,24} {
\draw[xshift=\x cm] \trou;}
\foreach \x in {0,6} {
\draw[color=red,xshift=\x cm] (3,-0.2) ..controls +(0.2,-0.5) and +(0.2,0.5) .. (3,-2.3);
\draw[dashed,color=red,xshift=\x cm] (3,-0.2) ..controls +(-0.2,-0.5) and +(-0.2,0.5) .. (3,-2.3);
\draw[color=blue,xshift=\x cm] (3,0)ellipse(1.6 and 0.8);}
\foreach \x in {11.9,17.9,23.9} {
\draw[color=red,xshift=\x cm] (3,-0.2) ..controls +(0.2,-0.5) and +(0.2,0.5) .. (3,-2.3);
\draw[dashed,color=red,xshift=\x cm] (3,-0.2) ..controls +(-0.2,-0.5) and +(-0.2,0.5) .. (3,-2.3);}
\foreach \x in {12.2,18.2,24.2} {
\draw[color=blue,xshift=\x cm] (3,-0.2) ..controls +(0.2,-0.5) and +(0.2,0.5) .. (3,-2.3);
\draw[dashed,color=blue,xshift=\x cm] (3,-0.2) ..controls +(-0.2,-0.5) and +(-0.2,0.5) .. (3,-2.3);}
\begin{scope}[yscale=0.8]
 \begin{scope} [xshift=38cm,yshift=3cm]
\draw (0,1.25) ..controls +(1,0) and +(-2,1) .. (4,2) ..controls +(2,-1) and +(-2,-1) .. (8,2) ..controls +(2,1) and +(0,1.5) .. (12,0);
\draw (0,-1.25) ..controls +(1,0) and +(-2,-1) .. (4,-2) ..controls +(2,1) and +(-2,1) .. (8,-2) ..controls +(2,-1) and +(0,-1.5) .. (12,0);
\end{scope}
\foreach \x/\y in {38/3,44/3,38/-3,38/9} {
\draw[xshift=\x cm,yshift=\y cm] \trou;}
\draw (44,-1.75) ..controls +(0.4,-1) and +(0.4,1) .. (44,-4.25);
\draw[dashed] (44,-1.75) ..controls +(-0.4,-1) and +(-0.4,1) .. (44,-4.25);
\begin{scope} [xshift=26cm,yshift=-3cm]
\draw (14,2) ..controls +(2,1) and +(-1.2,0) .. (18,1.25);
\draw (14,-2) ..controls +(2,-1) and +(-1.2,0) .. (18,-1.25);
\draw (12,1.25) ..controls +(0.5,0) and +(-1,-0.5) .. (14,2);
\draw (12,-1.25) ..controls +(0.5,0) and +(-1,0.5) .. (14,-2);
\end{scope}
\begin{scope} [xshift=26cm,yshift=9cm]
\draw (14,2) ..controls +(2,1) and +(-1,0) .. (18,2);
\draw (14,-2) ..controls +(2,-1) and +(-1,0) .. (18,-2);
\draw (12,1.25) ..controls +(0.5,0) and +(-1,-0.5) .. (14,2);
\draw (12,-1.25) ..controls +(0.5,0) and +(-1,0.5) .. (14,-2);
\end{scope}
\begin{scope} [xshift=26cm,yshift=-9cm]
\draw (12,1.25) ..controls +(2,0) and +(-2,0) .. (18,2);
\draw (12,-1.25) ..controls +(2,0) and +(-2,0) .. (18,-2);
\end{scope}
\foreach \y in {9,-9} {
\begin{scope} [xshift=44cm,yshift=\y cm]
\foreach \s in {1,-1} {
\draw (0,2*\s) ..controls +(0.2,-0.5*\s) and +(0.2,0.5*\s) .. (0,0.5*\s);
\draw[dashed] (0,2*\s) ..controls +(-0.2,-0.5*\s) and +(-0.2,0.5*\s) .. (0,0.5*\s);}
\draw (0,0.5) arc (90:270:0.5);
\end{scope}}
\foreach \y in {-6,0,6} {
\draw (38,1.75+\y) arc (90:270:1.75);}
\foreach \s in {1,-1} {
\draw (30,1.56*\s) .. controls +(3,0) and +(-5,0) .. (38,10.25*\s);}
\foreach \y in {-9,-3,3} {
\foreach \x/\c in {37.9/red,38.2/blue} {
\begin{scope} [xshift=\x cm,yshift=\y cm,\c]
\draw (0,1.25) ..controls +(0.4,-1) and +(0.4,1) .. (0,-1.25);
\draw[dashed] (0,1.25) ..controls +(-0.4,-1) and +(-0.4,1) .. (0,-1.25);
\end{scope}}}
\end{scope}
\end{tikzpicture}
\end{center} \caption{Heegaard diagram for $C_{i-1}\cup C_i$\\{\footnotesize Here, $P$ has four components: a closed genus--$2$ surface, an annulus, a punctured torus and a twice-punctured torus.}} \label{fig:StandardDiagram}
\end{figure}

\begin{proposition} \label{prop:diagrams}
 A multisection diagram determines a unique compact $4$--manifold up to diffeomorphism.
\end{proposition}
\begin{proof}
 Let $X$ and $X'$ be multisected $4$--manifolds with diffeomorphic multisection diagrams $(\Sigma;\alpha_1,\dots,\alpha_n)$ and $(\Sigma';\alpha'_1,\dots,\alpha'_n)$, meaning that there is a diffeomorphism $\varphi:\Sigma\to\Sigma'$ such that $\varphi(\alpha_i)=\alpha'_i$. First extend $\varphi$ along defining disk systems for the $3$--dimensional compression bodies of the multisection. Then extend it to the whole compression bodies (which amounts to extending a diffeomorphism from $S^2$ to $B^3$ or from $P\times\{1\}$ to $P\times I$). Finally extend $\varphi$ to the $4$--dimensional compression bodies using Theorem~\ref{th:LPrel}.
\end{proof}

\subsection{The bad case}

In this section, we discuss the failure of Proposition~\ref{prop:diagrams} in the case when the surface~$P$ is allowed to contain some $2$--spheres. Accordingly, we allow $2$--spheres in the negative boundary of a $3$--dimensional compression body. 
We will analyze an example pointed out by David Gay. 

Start with a torus $\Sigma=S^1\times S^1$ and three parallel essential simple closed curves $\alpha_i\subset\Sigma$. Form the product $\Sigma\times\Delta$ where $\Delta$ is a $2$--disk, see Figure~\ref{figXn}. Choose three distinct points $p_i\in\partial\Delta$, $i=1,2,3$, and glue a $4$--dimensional $2$--handle $D_i\times D^2$ along $\alpha_i\times\{p_i\}$ for each $i$. More precisely glue the handle along $A_i\times[x_i,y_i]$ where $A_i$ is a closed tubular neighborhood of $\alpha_i$ in $\Sigma$ and $[x_i,y_i]$ an interval around $p_i$ in~$\partial\Delta$, so that the core of the handle is glued along $\alpha_i\times\{p_i\}$, see Figure~\ref{figXn} for the order of the points on the circle. It remains to glue some $3$--handles.

\begin{figure}[htb] 
\begin{center}
\begin{tikzpicture}
\begin{scope} [scale=0.7]
 \foreach \s in {-1,1}
 \draw (0,0) .. controls +(0,\s) and +(-1,0) .. (3,1.5*\s) .. controls +(1,0) and +(0,\s) .. (6,0);
 \draw (2,0) ..controls +(0.5,-0.25) and +(-0.5,-0.25) .. (4,0);
 \draw (2.3,-0.1) ..controls +(0.6,0.2) and +(-0.6,0.2) .. (3.7,-0.1);
 \draw[blue] (3,-0.2) ..controls +(0.2,-0.5) and +(0.2,0.5) .. (3,-1.5);
 \draw[dashed,blue] (3,-0.2) ..controls +(-0.2,-0.5) and +(-0.2,0.5) .. (3,-1.5)node[below] {$\scriptstyle{\alpha_2}$};
 \draw[red] (0.85,1) node[above] {$\scriptstyle{\alpha_1}$} .. controls +(0.5,0) and +(-0.3,0.5) .. (2.5,-0.03);
 \draw[red,dashed] (0.85,1) .. controls +(0.3,-0.5) and +(-0.5,0.1) .. (2.5,-0.03);
 \draw[green] (5.15,1) node[above] {$\scriptstyle{\alpha_3}$} .. controls +(-0.5,0) and +(0.3,0.5) .. (3.5,-0.03);
 \draw[green,dashed] (5.15,1) .. controls +(-0.3,-0.5) and +(0.5,0.1) .. (3.5,-0.03);
 \draw (3,0.8) node {$\scriptstyle{\widetilde\Sigma_1}$};
 \draw (1.5,-0.4) node {$\scriptstyle{\widetilde\Sigma_2}$};
 \draw (4.5,-0.4) node {$\scriptstyle{\widetilde\Sigma_3}$};
\end{scope}
\begin{scope} [xshift=8cm]
 \draw (0,0) node {$\scriptstyle{\Sigma\times\Delta}$} circle (1) circle (2);
 \foreach \t/\i in {150/1,270/2,30/3} {
 \foreach \s in {-25,25}
 \draw[rotate=\t+\s] (1,0) -- (2,0);
 \draw[rotate=\t] (1.5,0) node {$\scriptstyle{D_\i\times D^2}$};
 \draw[rotate=\t] (1,0) node {$\scriptscriptstyle{\bullet}$} (0.75,0) node {$\scriptstyle{p_\i}$};
 \draw[rotate=\t-25] (1,0) node {$\scriptscriptstyle{\bullet}$} (0.75,0) node {$\scriptstyle{x_\i}$};
 \draw[rotate=\t+25] (1,0) node {$\scriptscriptstyle{\bullet}$} (0.75,0) node {$\scriptstyle{y_\i}$};
 \draw[rotate=\t-60] (1.5,0) node {$\scriptstyle{B_\i}$};
 }
\end{scope}
\end{tikzpicture}
\caption{Decomposition of the surface $\Sigma$ and schematic of the construction of the manifolds $X_n$} \label{figXn}
\end{center}
\end{figure}

Define a foliation of $\Sigma$ by simple closed curves $\alpha_t$, $t\in[0,3]$, such that, for $i=1,2,3$, $\alpha_i$ is the curve previously defined, and $\alpha_0=\alpha_3$.
From now on, the indices $i$ are considered modulo $3$. For each $i$, let $t\in[0,1]\mapsto q_i(t)\in\partial\Delta$ be an injective path from $y_{i-1}$ to $x_i$, chosen so that $q_1$, $q_2$ and $q_3$ are pairwise disjoint. 
Set $\Sigma_i=\cup_{t\in[0,1]}\big(\alpha_{t+i-1}\times\{q_i(t)\}\big)\subset\Sigma\times\Delta$ (Figure~\ref{figXn} represents the projection $\widetilde\Sigma_i$ of $\Sigma_i$ on~$\Sigma$).
Now let $D_i^x$ be a disk parallel to $D_i$ on the boundary of the $2$--handle and attached to $\alpha_i\times\{x_i\}$. Similarly define~$D_{i-1}^y$. For $i=1,2$, glue a $3$--handle $B_i$ along the $2$--sphere $D_{i-1}^y\cup(-D_i^x)\cup\Sigma_i$. 

In the gluing of $B_3$, we shall give more flexibility in order to produce a family of distinct manifolds. Fix an integer $n$.
Define a path $q_3^n(t)$ from $y_2$ to $x_3$ in $\partial\Delta$ as a concatenation of $q_3$ and $n-1$ full turns around $\Delta$. 
Set $\Sigma_3^n=\cup_{t\in[0,1]}\big(\alpha_{t+2}\times\{q_3^n(t)\}\big)$, and glue a $3$--handle $B_3$ along $D_2^y\cup(-D_3^x)\cup\Sigma_3^n$.

We claim that this defines trisected manifolds with boundary $X_n$ sharing a common trisection diagram, namely that of Figure~\ref{figXn}. Since they are constructed by gluing handles, it is easy to compute their homology. We get $H_2(X_n)\cong\Z/n\Z$, so that the $X_n$ are non-homeomorphic manifolds.

This failure of Proposition~\ref{prop:diagrams} can be analyzed as follows. The $3$--dimensional pieces of the trisections we constructed are punctured solid tori, say $C_i$. These are reducible manifolds, and each curve $\alpha_i$ on their positive boundary bounds a family of pairwise non-isotopic properly embedded disks indexed by~$\Z$. This implies that they admit diffeomorphisms that restrict to the identity on the positive boundary, but are non-isotopic to the identity. The non-existence of such diffeomorphisms was a key point in our proof of the relative Laudenbach--Poénaru's theorem. One can check that the $X_n$ are related by the following move. Cut $X_n$ along one of the $C_i\setminus\partial_+C_i$ and reglue via a diffeomorphism of $C_i$ that restricts to the identity on the positive boundary, but is non-isotopic to the identity.

In fact, the boundary of $X_n$ is diffeomorphic to $S^1\times S^2$. Gluing a copy of $S^1\times B^3$ to $X_n$ provides a trisected closed $4$--manifold with the diagram of Figure~\ref{figXn} as a trisection diagram, thus diffeomorphic to~$S^1\times S^3$. It follows that $X_n$ is diffeomorphic to $S^1\times S^3$ with an open tubular neighborhood of a simple closed curve removed.

\def\cprime{$'$}
\providecommand{\bysame}{\leavevmode\hbox to3em{\hrulefill}\thinspace}
\providecommand{\MR}{\relax\ifhmode\unskip\space\fi MR }
\providecommand{\MRhref}[2]{%
  \href{http://www.ams.org/mathscinet-getitem?mr=#1}{#2}
}
\providecommand{\href}[2]{#2}

\end{document}